\documentclass{article}
\usepackage{amsmath}
\usepackage{amsthm}
\usepackage{amssymb}
\usepackage{mathrsfs}
\usepackage{amscd}
\usepackage{graphicx}
\usepackage{tikz}

\newtheorem{thm}{Theorem}[section]

\newtheorem{lem}[thm]{Lemma}

\newtheorem{prob}[thm]{Problem}

\newtheorem{obs}[thm]{Obersvation}
\newtheorem{claim}{Claim}

\title{Maximizing the number of independent sets of fixed size in $K_n$-covered graphs~\thanks{The work was supported by NNSF of China (No. 11671376) and Anhui Initiative in Quantum Information Technologies (AHY150200).} }
\author{Anyao Wang$^a$, \quad Xinmin Hou$^b$,\quad Boyuan Liu$^c$, \quad Yue Ma$^d$\\
\small $^{a,b,c,d}$ Key Laboratory of Wu Wen-Tsun Mathematics\\
\small School of Mathematical Sciences\\
\small University of Science and Technology of China\\
\small Hefei, Anhui 230026, China.
}

\begin{document}
\maketitle

\begin{abstract}
A graph $G$ is $H$-covered by some given graph $H$ if each vertex in $G$ is contained in a copy of $H$. In this note,  we give the maximum number of independent sets of size $t\ge 3$ in $K_n$-covered graphs of size $N\ge n$ and determine the extremal graph for $N\ge n+t-1$.  The result answers a question proposed by  Chakraborti and Loh. The proof uses an edge-switching operation of hypergraphs which remains the number of independent sets nondecreasing.

\end{abstract}

\section{Introduction}

Let $H$ be a given graph. A  graph $G$ is called {\it $H$-covered} if every vertex of $G$ is contained in at least one copy of $H$.
The extremal problems under $H$-covered condition were studied systematically by Chakraborti and Loh in~\cite{CL19}. They completely solved the problem of minimizing the number of edges in an $H$-covered graph with given number of vertices when $H$ is a clique or more generally when $H$ is a regular
graph with degree at least about half its number of vertices. 
We write $i_t(G)$ and $k_t(G)$ for the number of independent sets and cliques of size $t$ in graph $G$, respectively. So minimizing $k_2(G)$ is equivalent to maximizing $i_2(G)$. But for $t>2$, as pointed by Chakraborti and Loh in~\cite{CL19}, the situation is quite different from $t=2$, so they proposed the following question.
\begin{prob}[\cite{CL19}]\label{PROB: p1}
It will be interesting to consider the problem of maximizing the number of independent
sets of order $t>2$ in an $n$-vertex $H$-covered graph.
\end{prob}
The problem of determining the number of independent sets of fixed size under $H$-covered condition has been investigated in literatures. In what follows,  let $G_1\vee G_2$ be the join of graphs $G_1$ and $G_2$. For $H$ being  a star $K_{1,d}$, Engbers and Galvin~\cite{EG14} showed that every $n$-vertex $K_{1,d}$-covered graph $G$ with every vertex being the root of $K_{1,d}$ has $i_t(G)\le i_t(K_{d,n-d})$, and the equality holds if and only if $G=D\vee \overline{K_{n-d}}$ for most of cases of $t$ larger than $d$ and $n\ge 2d$, where $D$ is any graph on $d$ vertices and $\overline{K_{n-d}}$ is the empty graph on $n-d$ vertices. The result supports a conjecture of Galvin given in~\cite{G11}, and Engbers and Galvin further conjectured that the result holds for all positive integers $n, t, d$ with $n\ge 2d$ and $t\ge 3$. This conjecture was solved completely by Gan, Loh and Sudakov in~\cite{GLS15}, in which they count cliques instead of independent sets in the complementary graph as Cutler and Radcliffe did in~\cite{CR17}. In this note, we consider the problem of maximizing the number of independent sets in a $K_n$-covered graph with given number of vertices, our result answers Problem~\ref{PROB: p1} completely when $H$ is a clique.
The extremal graph has similar structure as $H$ being a star.
For given positive integers $n,k$ with $n\ge k$, write $S_{n,k}=K_k\vee \overline{K_{n-k}}$, the join graph of the complete graph $K_k$ on $k$ vertices  and the empty graph $\overline{K_{n-k}}$ on $n-k$ vertices. The following is our main result.

\begin{thm}\label{THM: main}
For any positive integers $n, t, N$ with $t\ge{3}$ and $N\ge{n}$, every $K_{n}$-covered graph $G$ on $N$ vertices satisfies $i_t(G)\le {N-n+1\choose t}$, and when $N\ge n+t-1$, $S_{N, n-1}$ is the unique extremal graph.
\end{thm}	

The following theorem given by Chakraborti and Loh~\cite{CL19} 
will be used in the proof of our main result.

\begin{thm}[Chakraborti and Loh~\cite{CL19}]\label{THM: CL}
	For any positive integers $q, n, t$ with $2\le{t}\le{n}$, and integer $N=qn+r$ with $0\le{r}\le{n-1}$, the graph consisting of 2 copies of $K_n$ sharing $n-r$ vertices, together with the disjoint union of $q-1$ many $K_n$, has the least number of copies of $K_t$ among all $K_n$-covered graphs on $N$ vertices. Moreover, this is the unique such graph.
\end{thm}

\noindent{\bf Remark A:} Chakraborti and Loh remarked in~\cite{CL19} that: we have some initial observations
that the structure of the optimal graph for Problem~\ref{PROB: p1} might be drastically different for $t>2$.



We will use an edge-switching operation to edges of a hypergraph as our main tool in the proof of Theorem~\ref{THM: main}. In what follows, we give the standard  definitions and notation.
A {\it hypergraph} is a pair $H=(V,E)$, where $V$ is a set of elements called vertices, and $E$ is a collection of subsets of $V$ called edges.
A {hypergraph} $H=(V,E)$ is called an $r$-uniform hypergraph, or $r$-graph, if each edge of $E$ has uniform size $r$.
In this article, all hypergraphs $H$ considered are simple, i.e. $H$ contains no multiple edges.
We call $|V|$ {\it the order} of $H$ and $|E|$ {\it the size} of $H$, also denoted by $|H|$ or $e(H)$.  So a graph is a $2$-uniform hypergraph by definition and we write graph for 2-graph for short.
Given $S\subseteq V(H)$, the {\it degree} of $S$, denote by $d_H(S)$, is the number of edges of $H$ containing $S$. The minimum $s$-degree $\delta_s(H)$ of $H$ is the minimum of $d_H(S)$ over all $S\subseteq V(H)$ of size $s$.
We call $\delta_1(H)$ the \emph{minimum degree} of $H$, that is $\delta_1(H)=\min\{d_H(v) : v\in V(H)\}$. Let $N_H(S)=\{ T : S\cup T\in E(H)\}$ and $N_H[S]=N_H(S)\cup\{S\}$.
The {\it $s$-shadow} of a hypergraph $H$ is an $s$-uniform hypergraph $L$ on vertex set $V(H)$ and an $s$-set $S\in E(L)$ if and only if there is an edge $e\in E(H)$ containing $S$.
An {\it independent set} $I$ in ${H}$ is a set of vertices such that $|I\cap{e}|\le{1}$.
Let $I_t(H)$ be the set of independent sets of size $t$ in $H$ and $i_t(H)=|I_t(H)|$ defined as above.
Given $S, T\subseteq V(H)$, write $G[S]$ for the  subgraph induced by $S$ and $E_G(S,T)$ for the set of edges with one end in $S$ and one in $T$.
Given two integers $a,b$ with $a<b$, write $[a,b]$ for the set $\{a, a+1, \ldots, b\}$ and $[b]$ for $[1,b]$. Given two sets of $A, B$, write $A-B=A\setminus(A\cap B)$.

A $K_n$-covered graph $G$ has a natural way associated with an $n$-uniform hypergraph $\mathcal{G}$ on vertex set $V(G)$  and edge set $E(\mathcal{G})=\{ e : e\subseteq V(G), |e|=n \mbox{ and } G[e]\cong K_n\}$. So the 2-shadow of $\mathcal{G}$ is a spanning $K_n$-covered  subgraph of $G$ and $i_t(G)\le i_t(\mathcal{G})$.
A  graph $G$ is called {\it edge-critical} $K_n$-covered if $G$ is $K_n$-covered but for any edge $e\in{E(G)}$, $G-{e}$ is not $K_n$-covered. For example, $S_{N,n-1}$ is an edge-critical $K_n$-covered graph.
Clearly, each edge of an edge-critical graph is contained in a copy of $K_n$. 
For edge-critical $K_n$-covered graphs, we have the following observation.
\begin{obs}\label{OBS: o1}
Let $G$ be an edge-critical $K_n$-covered graph and $\mathcal{G}$ be its associated hypergraph. Then the  following hold:

(1)  the 2-shadow of $\mathcal{G}$ is isomorphic to $G$;

(2) $i_t(G)=i_t(\mathcal{G})$;

(3) removing any hyperedge from $\mathcal{G}$ gives rise to at least one isolated vertex;

(4) $\delta(G)=n-1$.
\end{obs}
\begin{proof}
(1) and (2) come directly from the definitions of $G$ and $\mathcal{G}$.

(3) If not, then there is hyperedges $e_0$ and $e_1,\ldots, e_k$ such that $e_0\subseteq e_1\cup\ldots\cup e_k$ and $e_0\cap e_i\not=\emptyset$ for each $i\in[k]$. Since $\mathcal{G}$ is $n$-uniform, $k\ge 2$. So $E_G(e_1, \cup_{i=2}^ke_i)\not=\emptyset$. But $G-E_G(e_1, \cup_{i=2}^ke_i)$ is still $K_n$-covered, a contradiction to the edge-criticality of $G$.

(4) Suppose to the contrary that $\delta(G)\ge{n}$. Then by the edge-criticality of $G$, each vertex is contained in at least two copies of $K_n$.
Choose a copy of $K_n$ in $G$, say $H_0$, and let $V(H_0)=\{u_1,\ldots,u_n\}$. For each $u_i$, there is a copy of $K_n$, say $H_i$, covering $u_i$ but  $H_i\neq{H_0}$ (it is possible that $H_i=H_j$ for $i\not=j$).
So $V(H_0)\subseteq{V(H_1)\cup{\ldots}\cup{V(H_n)}}$, i.e. the hyperedge corresponding to $H_0$ in the associated hypergraph $\mathcal{G}$ of $G$ is covered in the union of the hyperedges corresponding to $H_1,\dots, H_n$, a contradiction to (3).

\end{proof}




The rest of the note is arranged as follows. In Section 2, we give some lemmas and the proof of Theorem~\ref{THM: main}. We give some discussion and remark in the last section.

\section{Proof of Theorem~\ref{THM: main}}

The proof uses an novel technique, called the edge-switching operation, the graph version has been used by Fan in~\cite{Fan02}  and Gao, Hou in~\cite{Gao19} to cope with problems related to cycles in graphs. Here we give a hypergraph version of the edge-switching operation.
For a hypergraph $\mathcal{H}$, define $f(\mathcal{H})=\prod\limits_{v\in V(\mathcal{H})}d_1(v)$ where $d_1(v)$ is the 1-degrees of $v$ in $\mathcal{H}$.
\begin{lem}[Edge-switching lemma]\label{LEM: ES}
Let $\mathcal{H}$ be a hypergraph (not necessarily uniform) on $N$ vertices and $e_0$ be an edge of size $n$ in $E(\mathcal{H})$.
Fix the order of elements in $e_0$, e.g. denote $e_0=(v_1, v_2, \ldots, v_n)$.
Let $\{e_1, \ldots, e_k\}$ be the set of edges adjacent to $e_0$ and $n_i=|e_i\cap{e_0}|$ for $1\le{i}\le{k}$.
The {\it edge-switching} operation with respect to $e_0$ is defined as follows: for each ${i}\in [k]$, we replace the edge $e_i$ by the edge $e_i'=\{v_1, \ldots, v_{n_i}\}\cup (e_i-e_0)$.
Let $\mathcal{H}'$ be the resulting hypergraph after the edge-switching operation.  Then the following holds.

(a)  $i_3(\mathcal{H})\le{i_3(\mathcal{H}')}$, the equality holds if and only if $\mathcal{H}'\cong\mathcal{H}$.

(b) $f(\mathcal{H}')\le f(\mathcal{H})$, the equality holds if and only if $\mathcal{H}'\cong\mathcal{H}$.
\end{lem}
\begin{proof}
By the definition of edge-switching operation, we have $e_0\cup e_1\cup\ldots\cup e_k=e_0\cup e_1'\cup\ldots\cup e_k'$ and $\{e_1', \ldots, e_k'\}$ is the set of edges adjacent to $e_0$ in $\mathcal{H}'$. Let $S=e_0\cup e_1\cup\ldots\cup e_k$ and $\overline{S}=V(\mathcal{H})\setminus S$.

(a) 
We partition $I_3(\mathcal{H})$  according to the positions of the elements of an independent set into four subsets.
Let $$T_1=\{I\in I_3(\mathcal{H}) : I\cap e_0=\emptyset\},$$
$$T_2=\{I\in I_3(\mathcal{H}) : |I\cap e_0|=1 \mbox{ and } |I\cap \overline{S}|=2\},$$
$$T_3=\{I\in I_3(\mathcal{H}) : |I\cap e_0|=1 \mbox{ and } |I\cap \overline{S}|=1\},$$
and $$T_4=\{I\in I_3(\mathcal{H}) : |I\cap e_0|=1 \mbox{ and } |I\cap \overline{S}|=0\}.$$
Then $T_1, T_2, T_3, T_4$ forms a partition of $I_3(\mathcal{H})$.
We partition $I_3(\mathcal{H}')$ into four subsets  $T_1', T_2', T_3', T_4'$ in the same way.
To show (a), it is sufficient to show that $|T_i|\le |T_i'|$ for $i=1,2,3,4$.
In the following proof, for $A\subseteq V(\mathcal{H})$, we write $\mathcal{H}-A$ for the hypergraph obtained from $\mathcal{H}$ by deleting the hyperedges $e$ included in $A$ and the vertices in $A$, i.e. $V(\mathcal{H}-A)=V(\mathcal{H})\setminus A$ and $E(\mathcal{H}-A)=\{e-A : e\in E(\mathcal{H})\}$; for a set $B$ of hyperedges not included in $\mathcal{H}$, write $\mathcal{H}\cup B$ for the hypergraph by adding the hyperedges of $B$.

For $i=1$, if an $I\in T_1$ then $I$ is also an independent set of $\mathcal{H}'$ by the definition of $\mathcal{H}'$, and vice versa. So we have $|T_1|=|T_1'|$. 

For $i=2$, an independent set $I\in T_2$ consists of one vertex in $e_0$ and an independent set in $I_2(\mathcal{H}-S)$. So $|T_2|=n|I_2(\mathcal{H}-S)|$.
Clearly, $I_2(\mathcal{H}-S)=I_2(\mathcal{H}'-S)$. Therefore, $|T_2'|=n|I_2(\mathcal{H}-S)|=|T_2|$.

Now we let $\tilde{e}_i=e_i-(e_0\cup e_1\cup\ldots\cup{e_{i-1}})$ for $i\in [k]$. Then $\tilde{e}_1, \ldots, \tilde{e}_k$ are pairwise disjoint and form a partition of $(e_1\cup\ldots\cup{e_k})-e_0$.
By the definition of $e'_i$, $e'_i-e_0=e_i-e_0$ for $i\in [k]$. So $\tilde{e}_i=e'_i-(e_0\cup e'_1\cup\ldots\cup{e'_{i-1}})$ for $i\in [k]$.

For $i=3$, an independent set $I\in T_3$ consists of one vertex in $e_0-e_j$ and an independent set in $I_2((\mathcal{H}-S)\cup \{\tilde{e}_j\})\setminus I_2(\mathcal{H}-S)$ for some $1\le j\le k$.
For each $j\in[k]$, let $$L_j=\{ I\in T_3 : |I\cap e_0|=1 \mbox{ and } |I\cap \tilde{e}_j|=1 \}.$$
Then $L_1, \ldots, L_k$ form a partition of $T_3$.
Similarly, an independent set $I\in T'_3$ consists of one vertex in $e_0-e'_j$ and an independent set in $I_2((\mathcal{H}'-S)\cup \{\tilde{e}_j\})\setminus I_2(\mathcal{H}'-S)$ for some $1\le j\le k$.
Define $$L'_j=\{ I\in T'_3 : |I\cap e_0|=1 \mbox{ and } |I\cap \tilde{e}_j|=1 \}$$ for each $j\in[k]$.
Then $L'_1, \ldots, L'_k$ form a partition of $T'_3$.
Since  $I_2(\mathcal{H}-S)=I_2(\mathcal{H}'-S)$, we have $I_2((\mathcal{H}'-S)\cup \{\tilde{e}_j\})=I_2((\mathcal{H}-S)\cup \{\tilde{e}_j\})$. And since $|e_0-e_j|=|e_0-e'_j|=n-n_j$, we have $|L_j|=|L_j'|$ for each $j\in [k]$. Therefore, $|T_3|=\sum\limits_{j=1}^k|L_j|=\sum\limits_{j=1}^k|L'_j|=|T'_3|$.

For $i=4$, an independent set $I\in T_4$ consists of one vertex in $e_0-(e_j\cup e_\ell)$ and an independent set in $I_2(\mathcal{H}[\tilde{e}_j\cup\tilde{e}_\ell]\cup \{\tilde{e}_j,\tilde{e}_\ell\})$ for some $1\le j<\ell\le k$. Let
$$M_{j\ell}=\{ I\in T_4 : |I\cap e_0|=1 \mbox{ and } |I\cap (\tilde{e}_j\cup\tilde{e}_\ell)|=2 \},$$
for $1\le j<\ell\le k$. Then $\{M_{j\ell} : 1\le j<\ell\le k\}$ forms a partition of $T_4$.
Similarly, an independent set $I\in T'_4$ consists of one vertex in $e_0-(e'_j\cup e'_\ell)$ and an independent set in $I_2(\mathcal{H}'[\tilde{e}_j\cup\tilde{e}_\ell]\cup \{\tilde{e}_j,\tilde{e}_\ell\})$ for some $1\le j<\ell\le k$. Define
$$M'_{j\ell}=\{ I\in T_4 : |I\cap e_0|=1 \mbox{ and } |I\cap (\tilde{e}_j\cup\tilde{e}_\ell)|=2 \},$$
for $1\le j<\ell\le k$. Then $\{M'_{j\ell} : 1\le j<\ell\le k\}$ forms a partition of $T'_4$ too.
By the definition of $\mathcal{H}'$, $\mathcal{H}'[\tilde{e}_j\cup\tilde{e}_\ell]=\mathcal{H}[\tilde{e}_j\cup\tilde{e}_\ell]$.
So $I_2(\mathcal{H}'[\tilde{e}_j\cup\tilde{e}_\ell]\cup \{\tilde{e}_j,\tilde{e}_\ell\})=I_2(\mathcal{H}[\tilde{e}_j\cup\tilde{e}_\ell]\cup \{\tilde{e}_j,\tilde{e}_\ell\})$. But $|e_0-(e'_j\cup e'_\ell)|=n-\max\{n_j,n_\ell\}\ge |e_0-(e_j\cup e_\ell)|$. So $|M_{j\ell}|\le |M'_{j\ell}|$ for any pair $1\le j<\ell\le k$.
Therefore, $$|T_4|=\sum\limits_{1\le j<\ell\le k}|M_{j\ell}|\le \sum\limits_{1\le j<\ell\le k}|M'_{j\ell}|=|T'_4|,$$
the equality holds if and only if $\mathcal{H}'\cong \mathcal{H}$.
This completes the proof of (a).

(b)  Without loss of generality, assume that $n_1\ge{n_2}\ge\dots\ge{n_k}$. For simplicity, we write $d_1(v)$ and $d_1'(v)$ for the 1-degrees of $v$ in $\mathcal{H}$ and $\mathcal{H}'$, respectively. So, we have
$$f(\mathcal{H})=\prod\limits_{v\in V(\mathcal{H})}d_1(v)=\prod\limits_{v\in e_0}d_1(v)\prod\limits_{v\in V(\mathcal{H})\setminus{e_0}}d_1(v)$$
and
$$f(\mathcal{H}')=\prod\limits_{v\in V(\mathcal{H}')}d'_1(v)=\prod\limits_{v\in e_0}d'_1(v)\prod\limits_{v\in V(\mathcal{H}')\setminus{e_0}}d'_1(v).$$
Since $d_1(v)=d'_1(v)$ for each $v\in V(\mathcal{H})\setminus e_0=V(\mathcal{H}')\setminus e_0$, to show $f(\mathcal{H}')<f(\mathcal{H})$, it is sufficient to show $\prod\limits_{v\in e_0}d'_1(v)<\prod\limits_{v\in e_0}d_1(v)$.
Define $g(x_1, \ldots, x_n)=\prod_{i=1}^{n}x_i$.
Then $\prod\limits_{v\in e_0}d_1(v)=g(d_1(v_{\pi(1)}), \ldots, d_1(v_{\pi(n)}))$, where $\pi$ is a permutation of $[n]$ such that $k+1\ge d_1(v_{\pi(1)})\ge\ldots\ge{d_1(v_{\pi(n)})}\ge 1$.
Similarly, we have $\prod\limits_{v\in e_0}d'_1(v)=g(d'_1(v_1),  \ldots, d'_1(v_n))=g(k+1, \dots, k+1, k, \dots, k, \dots, 1, \dots, 1)$ where the number of $k+1$ is $n_k$, the number of $i$ is $n_{i-1}-n_i$ for $2\le{i}\le{k}$ and the number of 1 is $n-n_1$.
Let $y_i^0=d_1(v_{\pi(i)})$ and $\alpha_i^0=d'_1(v_i)-y_i^0$ for $1\le{i}\le{n}$, i.e.
$(\alpha_1^0, \ldots, \alpha_n^0)=(d'_1(v_1),  \ldots, d'_1(v_n))-(y_1^0, \ldots, y_n^0)$.
Let$$\zeta_{ij}=\begin{cases}
 1, & \text{if } v_{\pi(i)}\in{e_j},\\
 0, & \text{if } v_{\pi(i)}\notin{e_j}.
 \end{cases}$$
Then, for each $m\in [n]$, $\sum\limits_{i=1}^m\zeta_{ij}\le m$, and furthermore, $\sum\limits_{i=1}^m\zeta_{ij}\le n_j$ if $m\ge n_j$.  So, for each $m\in[n]$, assume $n_{j}\le{m}<{n_{j-1}}$ for some $2\le{j}\le{k+1}$. By the double-counting argument, we have
\begin{eqnarray*}
\sum_{i=1}^my_i^0&=&\sum_{i=1}^m\left(\sum_{s=1}^k\zeta_{is}+1\right)\\
                 &=&\sum_{s=1}^k\sum_{i=1}^m\zeta_{is}+m\\
                 &\le& m(j-1)+{\sum_{s=j}^kn_{s}+m}\\
                 &=&{\sum_{i=1}^md'_1(v_{i})}.
\end{eqnarray*}
The same inequality holds for the cases $m<n_k$ or $m\ge{n_1}$.
Thus, 
\begin{equation*}
\sum_{i=1}^j\alpha_i^0\ge0 \text{ for each } {j}\in [{n}], \text{ and } \sum_{i=1}^n\alpha_i^0=0.
\end{equation*}
Let $m$ be the smallest index such that $\alpha_m^0\neq{0}$. Since  $\sum_{i=1}^m\alpha_i^0\ge0$, we have $\alpha_m^0>0$. Let $s$ be the smallest index such that $\alpha_s^0<{0}$. Such $s$ exits since $\sum_{i=1}^n\alpha_i^0=0$, and $m<s$. Let $$\begin{cases}
	y_i^1=y_i^0, & \text{if } i\neq{m, s},\\
	y_m^1=y_m^0+1,\\
    y_s^1=y_s^0-1.
	\end{cases}$$
Since $m<s$, $y_m^0\ge{y_s^0}$. So $g(y_1^1, \dots, y_n^1)<g(y_1^0, \dots, y_n^0)$. We claim that $y_1^1\ge\ldots\ge y_n^1$. Otherwise, we have $y_m^0+1=y_m^1\ge{y_{m-1}^1+1}=y_{m-1}^0+1$ or $y_s^0-1=y_s^1\le{y_{s+1}^1-1}$. Without loss of generality, assume  the former holds. Then $y_m^0=y_{m-1}^0$. Thus $d'_1(v_m)=y_m^0+\alpha_m^0>y_m^0=y_{m-1}^0=d'_1(v_{m-1})$, a contradiction.
Define $\alpha_i^1=d'_1(v_i)-y_i^1$ for $i\in[n]$, i.e. $(\alpha_1^1, \ldots, \alpha_n^1)=(d'_1(v_1),  \ldots, d'_1(v_n))-(y_1^1, \ldots, y_n^1)$.
By the definition of $m$ and $s$, we can easily check that
$$\sum_{i=1}^j\alpha_i^1\ge0 \mbox{  for each  $j\in[n]$, } \sum_{i=1}^n\alpha_i^1=0, \text{ and } \sum_{i=1}^n|\alpha_i^1|=\sum_{i=1}^n|\alpha_i^0|-2.$$
So we can continue the process to obtain sequences $(y_1^2, \ldots, y_n^2)$, $(\alpha_1^2, \ldots, \alpha_n^2)$, $\ldots$, $(y_1^t, \ldots, y_n^t)$, $(\alpha_1^t, \ldots, \alpha_n^t)$ such that $g(y_1^i, \ldots, y_n^i)<g(y_1^{i-1}, \ldots, y_n^{i-1})$ for $1\le i\le t$. We will stop at step $t$ if  $(\alpha_1^t, \dots, \alpha_n^t)=(0, \ldots, 0)$, such $t$ exits since $\sum_{i=1}^n|\alpha_i^0|$ is finite and decreases by 2 in each iteration.
Therefore,
$$\prod\limits_{v\in e_0}d'_1(v)=g(y_1^t, \ldots, y_n^t)<g(y_1^0, \ldots, y_n^0)=\prod\limits_{v\in e_0}d_1(v)$$
for $t>0$. If $t=0$ then $\alpha_i^0=0$ for each $i\in [n]$. So $d_1(v_{\pi(i)})=y_i^0=d'_1(v_i)$ for $i\in [n]$. Therefore, $\prod\limits_{v\in e_0}d'_1(v)=\prod\limits_{v\in e_0}d_1(v)$ if and only if $\mathcal{H}'\cong\mathcal{H}$.
This completes the proof of (b).

\end{proof}

A hypergraph $\mathcal{H}$ is called {\it stable} under edge-switching if the hypergraph obtained from the edge-switching operation with respect to any edge $e$ of $\mathcal{H}$ is isomorphic to $\mathcal{H}$. For example, the $n$-uniform hypergraph associated with the $K_n$-covered graph $S_{N, n-1}$ is stable under edge-switching operation.
\begin{lem}\label{LEM: Stable}
Let $\mathcal{H}$ be a stable hypergraph. If $\mathcal{H}$ is connected then $\Delta_1(H)=|E(\mathcal{H})|$.
\end{lem}
\begin{proof}
Let $x$ be the vertex with maximum 1-degree in $V(\mathcal{H})$ and let $e_1, \ldots, e_t$ be all the edges containing $x$. If $t=|E(\mathcal{H})|$ then we are done. Now suppose $t<|E(\mathcal{H})|$. Since $\mathcal{H}$ is connected, there is an edge $e_{t+1}\in E(\mathcal{H})$ such that $x\notin e_{t+1}$ but $e_{t+1}\cap e_j\not=\emptyset$ for some $j\in [t]$. Without loss of generality, assume $e_{t+1}\cap e_t\not=\emptyset$. For each $i\in[t-1]$, since $x\in e_i\cap e_t$ but $x\notin e_{t+1}$, we have $e_{t+1}\cap e_t\subseteq e_i\cap e_t$ because $\mathcal{H}$ is stable under edge-switching. Thus for each vertex $y\in e_{t+1}\cap e_t$, we have $d_1(y)\ge t+1>d_1(x)$, a contradiction.
\end{proof}

\noindent{\bf Remark B:} Note that $f(\mathcal{H})$ is finite for any finite hypergraph $\mathcal{H}$. So, by (b) of Lemma~\ref{LEM: ES}, we will obtain a stable hypergraph  from $\mathcal{H}$ after finite times of edge-switching operations with respect to edges of $\mathcal{H}$.

Now we are ready to give the proof of Theorem\ref{THM: main}.
\begin{proof}[Proof of Theorem~\ref{THM: main}:] We proof the theorem by induction on $t$. As mentioned in Remark A, the most difficult thing is to tackle the base case $t=3$.

\noindent{\bf The base case:} $t=3$.

We use induction on $n$.

For $n=1$,
it is a trivial case.
For $n=2$, $S_{N,1}\cong K_{N-1,1}$.
We show by induction on $N$ that $i_3(G)\le {N-1\choose 3}=i_3(S_{N,1})$ for any  $K_2$-covered graph $G$ on $N$ vertices, and $S_{N,1}$ is the unique extremal graph if $N\ge 4$. For $N\le4$, one can directly check the truth of the statement.
Now assume  $N\ge{5}$ and the result is true for all $K_2$-covered graphs of order less than $N$.
Let $G$ be a $K_2$-covered graph of order $N$. Without loss of generality, assume $G$ is edge-critical (otherwise, we may choose a minimum $K_n$-covered spanning subgraph $G'$ of $G$ to replace $G$. Clearly, $G'$ is edge-critical of order $N$ and $i_3(G')\ge i_3(G)$).
By (4) of Observation~\ref{OBS: o1}, $\delta(G)=1$.
If $G$ is 1-regular
then we can  calculate directly that $i_3(G)=8{N/2\choose3}<{N-1\choose 3}$.
Now assume $G$ is not 1-regular. Then we can choose $x\in V(G)$ such that $d_G(x)=1$  and $x$ has a neighbor $y$ of degree larger than one. Thus $G-x$ is also $K_2$-covered. Thus by induction hypothesis,  we have
$$i_3(G)=i_3(G-{x})+i_2(G-{N_G[x]})\le{{N-2\choose 3}+{	N-2\choose 2}}={ N-1\choose  3},$$
the equality holds if and only if $G-{x}\cong S_{N-1,1}$ and $G-{N_G[x]}\cong \overline{K_{N-2}}$, i.e., $G\cong S_{N,1}$.

Now assume $n>2$ and, for any $K_{n-1}$-covered graph $G'$ of order $N\ge n-1$, we have $i_3(G')\le {N-n+2\choose 3}$, and  the equality holds if and only if $G'\cong S_{N,n-2}$ for $N\ge n+1$.   Let $G$ be a $K_{n}$-covered graph on $N\ge n+2$ vertices. Suppose $i_3(G)\ge {N-n+1\choose 3}$. We also assume $G$ is edge-critical here (otherwise, we may choose an edge-critical spanning subgraph of $G$ to replace it).
\begin{claim}\label{CLA: c1}
	$G$ is connected.
\end{claim}
Suppose to the contrary that $G$ is disconnected. Since $G$ is $K_{n}$-covered, we have  $n-1\le{d_G(v)}\le{N-n-1}$ for each $v\in{V(G)}$ and $N\ge 2n$.
For $i\in\{0,1,2,3\}$, let $\tau_i$ be the set of unordered triples $\{u,v,w\}\subseteq V(G)$ such that  $G[\{u,v,w\}]$ has exactly $i$ edges.	
Counting the number of the triples in $\tau_1\cup \tau_2$,
we have  $$|\tau_1|+|\tau_2|=\frac{1}{2}\sum_{v\in{V(G)}}d_G(v)(N-1-d_G(v))\ge{\frac{(n-1)(N-n)N}{2}},$$
where the inequality holds since $N\ge 2n$.
Let $N=qn+r$ with $q\ge{2}$ and $0\le{r}\le{n-1}$. By Theorem~\ref{THM: CL}, $k_3(G)\le (q+1){n\choose 3}-{n-r\choose 3}$.	Then
\begin{eqnarray*}
	i_3(G)&=&{N\choose 3}-|\tau_1|-|\tau_2|-k_3(G)\\
	      &\le&{N\choose 3}-\frac{(n-1)(N-n)N}{2}-(q+1){n\choose 3}+{n-r\choose 3}\\
	&=&\frac{n^3q^3-(3n^3-3n^2r)q^2+(2n^3-6rn^2+3r^2n)q+3rn-3r^2}{6}.
\end{eqnarray*}
So
 $${N-n+1\choose 3}-i_3(G)\ge \frac{(n^3-n)(q-1)+r[r^2-3(n-1)r+3n^2-3n-1]}{6}. $$	
Let $f(r)=r^2-3(n-1)r+3n^2-3n-1$, $0\le r\le n-1$. Then $f'(r)=2r-3(n-1)<0$. Thus $f(r)$ is a decreasing function on $0\le r\le n-1$. So $f(r)\ge f(n-1)=n^2+n-3>0$. Therefore,
$${N-n+1\choose    3}-i_3(G)\ge \frac{(n^3-n)(q-1)+rf(r)}{6}>0    $$ for any $q\ge{2}$ and $n\ge{3}$, a contradiction to the assumption. So $G$ is connected.
	
Let $\mathcal{G}=(V, E)$ be the associated hypergraph with $G$.
Since $G$ is connected, $\mathcal{G}$ is connected.
Now we apply the edge-switching operations to $\mathcal{G}$ until we get a stable hypergraph $\mathcal{G}_0$. Let $G_0$ be the 2-shadow of $\mathcal{G}_0$.
Then $G_0$ is $K_n$-covered since $\mathcal{G}_0$ is $n$-uniform.
Since the edge-switching operation does not affect the connectivity, $\mathcal{G}_0$ is connected too. By Lemma~\ref{LEM: Stable}, $\mathcal{G}_0$ has a vertex $x$ of 1-degree $|E(\mathcal{G}_0)|$, or equivalently, $x$ has degree $|V(G_0)|-1=N-1$ in $G_0$.
Then there is no independent set $I\in I_3(G_0)$ containing $x$. Clearly, $G_0-{x}$ is $K_{n-1}$-covered. By induction hypothesis, we have  $i_3(G_0-{x})\le {
	N-n+1\choose 3}$.
By Lemma~\ref{LEM: ES}, $$i_3(G)=i_3(\mathcal{G})\le i_3(\mathcal{G}_0)=i_3(G_0)=i_3(G_0-{x})\le {N-n+1\choose 3},$$
the equality holds if and only if $\mathcal{G}\cong\mathcal{G}_0$ (or equivalently, $G\cong G_0$) and $G_0-x\cong S_{N-1,n-2}$, i.e. $G\cong S_{N,n-1}$.	This completes the proof of the base case.

Now assume $t\ge 4$  and $i_{t-1}(G)\le{|V(G)|-n+1\choose t-1}$ for any $K_n$-covered graph $G$ on at least $n$ vertices, and when $|V(G)|\ge n+t-1$, the equality holds if and only if $G\cong S_{|V(G)|, n-1}$. 	
Let $G$ be a $K_n$-covered graph on $N\ge n$ vertices.  We show by induction on $N$ that $i_{t}(G)\le{N-n+1\choose t}$, and when $N\ge n+t-1$ the equality holds if and only if $G\cong S_{N, n-1}$.
	
For $N\le n+t-1$, it can be easily check that the result is true.
Now assume $N\ge n+t$ and the result holds for all $K_n$-covered graphs of order at most $N-1$. Let $G$ be a $K_n$-covered graph on $N$ vertices and we may assume $G$ is edge-critical with the same reason as aforementioned.
By (4) of Observation~\ref{OBS: o1}, $\delta(G)=n-1$. Choose $v\in V(G)$ with $d_G(v)=n-1$ and let $S=\{u : u\in N_G[v]\text{ and }d_G(u)=n-1\}$.
Since $G$ is edge-critical, $S$ is a clique contained in a unique copy of $K_n$ in $G$.
So $G-S$ is still  $K_n$-covered. Let $s=|S|$ and let
$A=\{ I\in I_t(G) :  I\cap S=\emptyset \}$ and $B=\{I\in I_t(G) : |I\cap S|=1\}$. Then $|A|=i_t(G-S)$ and $|B|\le{s\cdot i_{t-1}(G-S)}$. So
\begin{eqnarray*}
	i_t(G)&=& |A|+|B|\\
          &\le& i_t(G-S)+ s\cdot i_{t-1}(G-S)\\
	      &\le & {N-s-n+1\choose t}+s\cdot {N-s-n+1\choose	t-1}\\
	&\le &{N-s-n+1\choose t}+\sum_{i=0}^{s-1}{N-s-n+1+i\choose	t-1}\\
	&=&{N-n+1\choose t},
\end{eqnarray*}
the equality holds if and only if $s=1$ and $G-S\cong S_{N-1, n-1}$, i.e. $G\cong S_{N, n-1}$.
\end{proof}

\section{Discussions and remarks}
In this note, we completely resolve Problem~\ref{PROB: p1} when $H=K_n$. From the result of Engbers and
Galvin~\cite{EG14}, and  Gan, Loh and Sudakov~\cite{GLS15}, we know that the optimal $K_{1,d}$-covered graphs $G$ of order $N$ maximizing $i_t(G)$ have structure like $D\vee \overline{K_{n-d}}$ with $|D|=d$. It will be interesting if one can show for which graph $H$, the optimal $H$-covered graphs $G$ of order $N$ maximizing $i_t(G)$ have structure like $D\vee \overline{K_{N-|H|+1}}$ with $|D|=|H|-1$.

\noindent{Remark:} Dr. Stijn Cambie told us that not all optimal graphs are of the form $D\vee \overline{K_{N-|H|+1}}$ for the above question, and he gave an example for $N=8$ and $H=C_6$ with the extremal $H$-covered graphs having no construction like $D\vee \overline{K_{N-|H|+1}}$.

\end{document}